\documentclass[11pt]{amsart}
\usepackage{fullpage}
\usepackage{amsmath}
\usepackage{amsfonts}
\usepackage{amsthm}
\usepackage{float}
\usepackage{amssymb}
\usepackage{amscd}
\usepackage{todonotes}
\newtheorem{thm}{Theorem}[section]
\newtheorem{lemma}[thm]{Lemma}
\newtheorem{prop}[thm]{Proposition}
\newtheorem{cor}[thm]{Corollary}

\newtheorem{question}[thm]{Question}
\theoremstyle{definition}
\newtheorem{definition}[thm]{Definition}
\theoremstyle{remark}
\newtheorem{example}[thm]{Example}
\newtheorem{remark}[thm]{Remark}

\newcommand{\N}{\mathbb N}

\title{Robust Graph Ideals}
\author[A.~Boocher, B.~Brown, T.~Duff, L.~Lyman, T.~Murayama, A.~Nesky, K.~Schaefer]{Adam Boocher, Bryan Christopher Brown, Timothy Duff, Laura Lyman, \\ Takumi Murayama, Amy Nesky and Karl Schaefer}

\begin{document}
\maketitle 
\begin{abstract}
Let $I$ be a toric ideal. We say $I$ is \emph{robust} if its universal Gr\"obner basis is a minimal generating set. We show that any robust toric ideal arising from a graph $G$ is also minimally generated by its Graver basis. We then completely characterize all graphs which give rise to robust ideals. Our characterization shows that robustness can be determined solely in terms of graph-theoretic conditions on the set of circuits of $G$.
\end{abstract}

\section{Introduction}
Let $A=({\bf a}_1 \vert {\bf a}_2 \vert \cdots \vert {\bf a}_m)$ be an $n \times m$ matrix with entries in $\N$. Consider the homomorphism $\phi\colon k[x_1,\ldots,x_m] \rightarrow k[s_1,\ldots, s_n]$ such that $x_i \mapsto {\bf s}^{{\bf a}_i}$, where by convention ${\bf s}^{\bf v} := s_1^{v_1}\cdots s_n^{v_n}$ for ${\bf v }= (v_1,\ldots,v_n) \in \N^n$. The \emph{toric ideal} $I_A$ is defined to be $\ker(\phi)$. Toric ideals arise naturally in several areas of study, including integer programming, algebraic statistics, geometric modeling, and graph theory (see \cite{BHP12, GP13, RTT12}).

It is well-known that toric ideals are prime ideals that are generated by binomials \cite[\S4]{Stu96}. Among these, distinguished sets of binomials in $I_A$ have been introduced and studied for many classes of toric ideals.  The \emph{Graver basis} of $I_A$, denoted $\mathcal{G}_A$, consists of all binomials which are \emph{primitive}; that is, all binomials ${\bf x}^{\bf c}-{\bf x}^{\bf d} \in I_A$ such that there does not exists a binomial  ${\bf x}^{\bf c'} - {\bf x} ^{\bf d'} \in I_A$ with ${\bf x}^{\bf c'} \mid {\bf x} ^{\bf c}$ and ${\bf x}^{\bf d'} \mid {\bf x} ^{\bf d}$. The \emph{universal Gr\"obner basis} of $I_A$, denoted $\mathcal{U}_A$, is the union of all reduced Gr\"obner bases for $I_A$. Furthermore, $\mathcal{U}_A$ is a Gr\"obner basis for $I_A$ under all monomial term orders.  Finally, we say that ${\bf x}^{\bf c}-{\bf x}^{\bf d}$ is a \emph{circuit} if it is irreducible and if the set of indices for which $c_i, d_i$ are nonzero is minimal with respect to inclusion. Let $\mathcal{C}_A$ be the set of all circuits in $I_A$. By a result of \cite{Stu96}, the inclusions $\mathcal{C}_A \subset \mathcal{U}_A \subset \mathcal{G}_A$ hold.

In this paper, we study toric ideals for which $\mathcal{U}_A$ is a minimal generating set for $I_A$.  We call these ideals \emph{robust}. In \cite{BR}, the authors classified all robust toric ideals generated by quadratics; however, there are significant obstacles in characterizing robustness for toric ideals generated in higher degrees. The purpose of this project is to characterize robustness for toric ideals arising from graphs, that is, when $A$ is the incidence matrix of a graph.   Our first main result shows that for graphs, robustness is strong enough to ensure that the universal Gr\"obner basis and Graver basis are equal:

\begin{thm}\label{intro:urobustiffgrobust}
Let $G$ be a simple graph.  Then $I_G$ is robust iff it is minimally generated by its Graver basis.
\end{thm}

This result is quite surprising, as it states that minimality of $\mathcal{U}_G$ implies that of $\mathcal{G}_G$.  This behavior was witnessed for some classes of hypergraph ideals in \cite{GP13} as well.    It is open whether or not this holds for general toric ideals.  The proof of Theorem \ref{intro:urobustiffgrobust} relies on characterizations of $\mathcal{U}_G$ and $\mathcal{G}_G$ given in \cite{TT11, Vil95}.  We then use graph-theoretic analysis of primitive binomials to complete the proof.  

Next, we characterize all graphs $G$ that give rise to robust ideals.  Given Theorem \ref{intro:urobustiffgrobust} this turns out to be equivalent to requiring that every primitive binomial is \emph{indispensable}--that is, it is contained in every set of minimal generators of $I_G$.  The following theorem is stated in terms of graph theoretic properties of the circuits of the graph $G$.

\begin{thm}\label{charcircuits}
  $I_G$ is robust if and only if the following conditions are satisfied.
  \begin{description}\itemsep0pt\parskip0pt\parsep0pt
    \item[R1] No circuit of $G$ has an even chord,
    \item[R2] No circuit of $G$ has a bridge,
    \item[R3] No circuit of $G$ contains an effective crossing, and
    \item[R4] No circuit of $G$ shares exactly one edge (and no other vertices) with another circuit such that the shared edge is part of a cyclic block in both circuits.
  \end{description}
\end{thm}

The layout of the paper is as follows: In Section 2, we review the construction of toric graph ideals, definitions relating to their study, and characterizations of circuits, the Graver basis, and the universal Gr\"obner basis of such an ideal. In Section 3 we prove that a robust toric graph ideal is minimally generated by its Graver basis, which facilitates major results in Section 4, where we present a graph-theoretic characterization of such ideals. 
In Section 5 we apply our results to list toric graph ideals generated in low degrees.  Finally we conclude with some open questions in the setting of general toric ideals.




\section{Toric Graph Ideals}
Let $G$ be a finite, simple graph with edge set $E$ and vertex set $V$. We assume that $G$ has no loops or multiple edges. We define the \emph{toric graph ideal} of $G$ to be the toric ideal associated with the homomorphism $\phi_G\colon k[E] \rightarrow k[V]$ such that $\phi_G(e_{ij})=v_i v_j$. Equivalently, $\phi_G$ sends an edge of $G$ to the product of its corresponding verticies. We denote such an ideal by $I_G$.

A \emph{walk} is a finite sequence of the form
\begin{equation*}
  w=(\{v_{i_1}, v_{i_2}\},\{v_{i_2}, v_{i_3}\},\ldots,\{v_{i_{n-1}}, v_{i_n}\})
\end{equation*}
with each $v_{i_j}\in V(G)$ and $e_{i_j}=\{v_{i_{j}}, v_{i_{j+1}}\}\in E(G)$. We denote a walk $w$ either by its sequence of edges, $(e_1,\cdots, e_k)$, or by its sequence of vertices, $(v_1,\cdots, v_k)$. A \emph{closed} walk is a walk with $v_{i_1} = v_{i_n}$.  A walk $w=(e_{i_1},e_{i_2},\dots,e_{i_n})$ is called \emph{even} (resp.~\emph{odd}) is $n$ is \emph{even} (resp.~\emph{odd}).

Given an even walk, $w=(e_1, e_2, \ldots, e_{2k})$, we denote its corresponding binomial
\begin{equation*}
  B_w = \prod_{i=1}^k e_{2i-1} - \prod_{i=1}^k e_{2i} \in I_G.
\end{equation*}
The ideal $I_G$ is generated by binomials corresponding to closed even walks; furthermore, every binomial generator arises through this correspondence \cite[Lem.~1.1]{OH99}. 
Given such a walk, let $w^+$ denote the set of edges with odd indices. Similarly, define $w^-$ to be the set of edges with even indices. We define edges of odd index as \emph{odd edges} and define \emph{even edges} analogously. Two edges are said to have the same \emph{parity} if they are both in $w^+$ or $w^-$. 

Many ideal-theoretic properties of $I_G$ can be interpreted graph theoretically.  To develop this relationship, we present some basic facts about robust toric ideals. 
The circuits of graph ideals have a combinatorial characterization.

\begin{definition}
A \emph{simple path} of a graph $G$ is a walk $(v_1, v_2, \ldots, v_k)$ such that the $v_i$ are all distinct.
\end{definition}

\begin{prop}[\cite{Vil95}]\label{circuitchar}
$B_w$ is a circuit of $I_G$ iff $w$ is one of the following:\\
(C1) an even cycle\\
(C2) two odd cycles joined at a single vertex\\
(C3) two vertex-disjoint odd cycles joined by a simple path $w = (v_1, v_2, \ldots, v_k)$ with $k > 1$ such that the intersection of $w$ with the first (resp. second) cycle is the first (last) vertex of $w$. 
\end{prop}

Circuits of the aformentioned types (and, by abuse of notation, the associated walks) will be referred to as C1, C2, and C3 circuits, respectively.
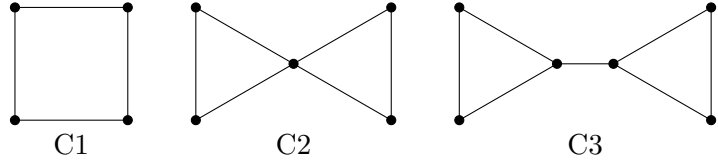
\begin{figure}[ht]
  \centering
  \begin{tikzpicture}[line cap=round,line join=round,x=1.5cm,y=1.5cm]
    \node (a) at (0,0) {\begin{tikzpicture}
      \coordinate (v1) at (0,0);
      \coordinate (v2) at (0,1);
      \coordinate (v3) at (1,1);
      \coordinate (v4) at (1,0);
      
      \draw (v1) -- (v2) -- (v3) -- (v4) -- (v1);
      
      \draw[fill=black] (v1) circle (0.04);
      \draw[fill=black] (v2) circle (0.04);
      \draw[fill=black] (v3) circle (0.04);
      \draw[fill=black] (v4) circle (0.04);
    \end{tikzpicture}};
    
    \node (b) at (a.east)[anchor=west,xshift=0.5cm] {\begin{tikzpicture}
      \coordinate (v1) at (0,0);
      \coordinate (v2) at (0,1);
      \coordinate (v3) at (0.866,0.5);
      \coordinate (v4) at (1.732,0);
      \coordinate (v5) at (1.732,1);
      
      \draw (v3) -- (v1) -- (v2) -- (v3) -- (v4) -- (v5) -- (v3);
      
      \draw[fill=black] (v1) circle (0.04);
      \draw[fill=black] (v2) circle (0.04);
      \draw[fill=black] (v3) circle (0.04);
      \draw[fill=black] (v4) circle (0.04);
      \draw[fill=black] (v5) circle (0.04);
    \end{tikzpicture}};
    
    \node (c) at (b.east)[anchor=west,xshift=0.5cm] {\begin{tikzpicture}
      \coordinate (v1) at (0,0);
      \coordinate (v2) at (0,1);
      \coordinate (v3) at (0.866,0.5);
      \coordinate (v3a) at (1.366,0.5);
      \coordinate (v4) at (2.232,0);
      \coordinate (v5) at (2.232,1);
      
      \draw (v3) -- (v1) -- (v2) -- (v3) -- (v3a) -- (v4) -- (v5) -- (v3a);
      
      \draw[fill=black] (v1) circle (0.04);
      \draw[fill=black] (v2) circle (0.04);
      \draw[fill=black] (v3) circle (0.04);
      \draw[fill=black] (v3a) circle (0.04);
      \draw[fill=black] (v4) circle (0.04);
      \draw[fill=black] (v5) circle (0.04);
    \end{tikzpicture}};
    
    \draw (a) node[anchor=north,yshift=-0.8cm] {C1};
    \draw (b) node[anchor=north,yshift=-0.8cm] {C2};
    \draw (c) node[anchor=north,yshift=-0.8cm] {C3};
  \end{tikzpicture}
  \caption{Examples of circuits of type C1, C2, and C3.}
\end{figure}
\par Similar conditions are needed for $B_w$ to be primitive.
\begin{prop}[{\cite[Lem.~3.2]{OH99}}]\label{ohprimchar}
If $B_w$ is primitive, then $w$ necessarily is of one of the following forms:\\
(P1) an even cycle\\
(P2) two odd cycles joined at a single vertex\\
(P3) $(c_1, w_1, c_2, w_2)$ where $c_1, c_2$ are vertex disjoint cycles and $w_1,w_2$ are walks which combine a vertex $v_1$ of $c_1$ and a vertex $v_2$ of $c_2$.
\end{prop}

\begin{figure}[ht]
  \centering
  \begin{tikzpicture}[line cap=round,line join=round,x=1.5cm,y=1.5cm]
      \coordinate (v1) at (0,0);
      \coordinate (v2) at (0,1);
      \coordinate (v3) at (0.866,0.5);
      \coordinate (v3a) at (1.366,0.5);
      \coordinate (v3b) at (1.866,1);
      \coordinate (v3c) at (2.366,0.5);
      \coordinate (v3d) at (2.866,0.5);
      \coordinate (v3e) at (1.866,0);
      \coordinate (v4) at (3.732,0);
      \coordinate (v5) at (3.732,1);
 
      \draw[black,thick] (v3) -- (v1) -- (v2) -- (v3);
      \draw[black,thick] (v3d) -- (v4) -- (v5) -- (v3d);
 
      \draw[color={rgb:red,125;green,33;blue,129},ultra thick] (v3) -- (v3a);
      \draw[color={rgb:red,125;green,33;blue,129},ultra thick] (v3c) -- (v3d);
 
      \draw[red,ultra thick] (v3a) edge node[anchor=south east]{$w_1$} (v3b);
      \draw[red,ultra thick] (v3b) -- (v3c);
      \draw[blue,ultra thick] (v3a) -- (v3e);
      \draw[blue,ultra thick] (v3e) edge node[anchor=north west]{$w_2$} (v3c);

      \draw[fill=black] (v1) circle (0.04);
      \draw[fill=black] (v2) circle (0.04);
      \draw[fill=black] (v3) circle (0.04) node[anchor=north]{$v_1$};
      \draw[fill=black] (v3a) circle (0.04);
      \draw[fill=black] (v3b) circle (0.04);
      \draw[fill=black] (v3c) circle (0.04);
      \draw[fill=black] (v3d) circle (0.04) node[anchor=north]{$v_2$};
      \draw[fill=black] (v3e) circle (0.04);
      \draw[fill=black] (v4) circle (0.04);
      \draw[fill=black] (v5) circle (0.04);
  \end{tikzpicture}
  \caption{Example of a P3 primitive walk.}
\end{figure}
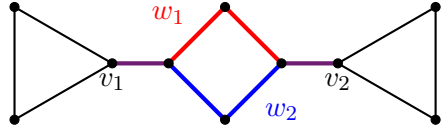
\par Primitive elements of the aforementioned types (and their associated walks) will be referred to as P1, P2, and P3 elements, respectively. Note that all P1 (resp.~P2) primitive elements are also C1 (resp.~C2) circuits, so primitive noncircuits must be of type P3.
\par However, there exist walks of the third type, P3, which give rise to non-primitive binomials. Necessary and sufficient conditions for primitivity require the introduction of new terminology. The definitions which follow are borrowed from \cite{TT11}.
\par A \emph{cut vertex} (resp. \emph{cut edge}) of a graph is a vertex (resp. edge) whose removal increases the number of connected components of a graph.  A graph is \emph{biconnected} if it is connected and does not contain a cut vertex. A \emph{block} is a maximal biconnected subgraph. A \emph{cyclic block} is a block which is 2-regular; namely, each vertex in the block is contained in exactly two edges. A \emph{sink} of a block is a common vertex of two odd or two even edges. 
\par While $G$ does not have multiple edges, we say an edge $e$ is a \emph{multiple edge of a walk} if $e$ appears more than once in the walk.  We define a walk to be \emph{strongly primitive} if it is primitive and does not contain two sinks within distance $1$ of each other in any cyclic block. These tools yield a graph-theoretic description of the Graver basis $\mathcal{G}_G$.

\begin{prop}[\cite{RTT12}]\label{primitivebinomial}
$B_w$ is primitive iff the following hold:\\
(1) every block of $w$ is a cycle or a cut edge\\
(2) every multiple edge is a cut edge and is traversed exactly twice \\
(3) every cut vertex of $w$ belongs to exactly 2 blocks and is a sink of both
\end{prop}
Still more care is required to describe $\mathcal{U}_G$ and the minimal generators. A cyclic block of a primitiive walk is \emph{pure} if all of its edges have the same parity. We have the following result from \cite{TT11}:  

\begin{prop}\label{ugbelementchar}
Let $w$ be a primitive walk. Then $B_w\in \mathcal{U}_G$ iff no cyclic block of $w$ is pure.
\end{prop}

In view of the containments $\mathcal{C}_G \subset \mathcal{U}_G \subset \mathcal{G}_G$ and the first two propositions, circuits of type C1 and C2 are always in $\mathcal{U}_G$. Then, by using Propositions \ref{primitivebinomial} and \ref{ugbelementchar}, we have the following result.

\begin{cor}\label{universalgroebnerbasiselt}
An element of the universal Gr\"{o}bner basis is either\\
(A) an even cycle\\
(B) two odd cycles joined at a single vertex\\
(C) a walk of the form $(c_1, w_1, c_2, w_2)$, where $c_1,c_2$ are vertex disjoint and $w_1, w_2$ are walks connecting them, subject to the conditions\\
\indent i) every block is a cycle or a cut edge\\
\indent ii) every multiple edge is a double edge and a cut edge\\
\indent iii) every cut vertex belonds to exactly 2 blocks and is a sink of both \\
\indent iv) no cyclic block is pure
\end{cor}

Finally, to study the minimal generators of a toric graph ideal, we must understand how the closed even walks relate to the larger graph. An edge $f\in E$ is said to be a \emph{chord} of a walk $w$ if both of its vertices belong to $w$ but $f$ itself does not. Chords fall into three classes.  A \emph{bridge} $f=\{v_1,v_2\}$ of a primitive walk $w = (e_1, e_2, \ldots, e_{2k})$ is a chord such that $w$ contains two different blocks $B_1$, $B_2$ with $v_1\in B_1$ and $v_2\in B_2$. A chord $f= \{v_i,v_j\}$ that is not a bridge is called \emph{even} (resp.~\emph{odd}) if the walks $(e_1, e_2, \ldots, e_{i-1}, f, e_j , e_{j+1}, \ldots, e_{2k})$ and $(e_i, e_{i+1}, \ldots, e_{j-1}, f)$ are both even (resp.~odd). Note that a chord starting at a cut vertex is always a bridge, since it is contained in two distinct blocks. Due to a result in \cite{RTT12}, binomials that occur in a minimal generating set are necessarily strongly primitive and contain no even chords or bridges. (See Section 4)

Let $w=((v_1,v_2),(v_2,v_3),\dots,(v_{2n}, v_1))$ be a primitive walk.  Let $f=(v_i,v_j)$ and $f'=(v_k,v_\ell)$ be two odd chords such that $j-i, \ell-k\in 2\N$, with $1\le i<j\le 2n$ and $1\le k<l\le 2n$. Then, $f$ and $f'$ \emph{cross effectively} if $i-k$ is odd and either $i<k<j<\ell$ or $k<i<\ell<j$. Note that if two odd chords $f$ and $f'$ cross effectively in $w$, then all their vertices are in the same
cyclic block of $w$. 

From here, if $w$ is a walk of $G$, we say $w^{-1}$ to denote $w$ traversed in the opposite direction. So, if $w=(e_1,\dots,e_n)$, then $w^{-1}=(e_n,\dots,e_1)$.

\section{Graver Bases and Robustness}
We say that the toric graph ideal $I_G$ is \emph{robust} if $\mathcal{U}_G$ is a minimal generating set for $I_G$.  We call a graph $G$ robust if $I_G$ is robust. Robustness is a relatively strong property as it ensures, for instance, that all initial ideals have the same minimal number of generators:
\begin{equation*}
  \mu(I_G) = \mu(\mathrm{in}_<I_G)~\text{for all term orders $<$.}
\end{equation*}
In terms of the binomials themselves we will use the following necessary condition:
\begin{lemma}\label{ugbnondividing}
If $I_G$ is robust, then no term of an element of $\mathcal{U}_G$ can divide a term of another element of $\mathcal{U}_G$.
\end{lemma}
\begin{proof}
  Suppose $\mathcal{U}_A$ contains binomials, 
  $$f = m_1 - m_2, \ g = n_1 - n_2 $$
  with $m_1$ dividing $n_1$.  Then some variable $x$ divides $m_1$ but not $m_2$ by primality of $I_G$.  Taking $<$ to be the Lex term order with $x$ first, we see that $(\mathrm{in_<}\:f) \mid (\mathrm{in_<}\:g).$ Thus, $\mu(\mathrm{in}_<I_G)) < \lvert \mathcal{U}_G \rvert = \mu(I_G)$, a contradiction.
\end{proof}

Our first main result states that the containment $\mathcal{U}_G\subset \mathcal{G}_G$ is an equality if $G$ is a robust graph.

\begin{thm}\label{urobustiffgrobust}
$I_G$ is robust iff it is minimally generated by its Graver basis
\end{thm}
\begin{proof}
If $I_G$ is minimally generated by $\mathcal{G}_G$ then since $\mathcal{U}_G\subset \mathcal{G}_G$, and both generate $I_G$, it follows that $\mathcal{U}_G$ is also a minimal generating set.  Hence $I_G$ is robust. 

To prove the other direction we will prove the contrapositive.  We assume that there is a primitive walk $w$ of $G$ whose corresponding binomial $b_w = w^+ - w^-$ is not in $\mathcal{U}_G$, then we construct another primitive walk $w'$ whose binomial is in $U_G$ but is not minimal, so that $G$ is not robust.

\par Let such a $w$ as above be given. Since $b_w \notin \mathcal{U}_G$, by Corollary \ref{universalgroebnerbasiselt}, $w$ must contain at least one pure cyclic block $B$. First, we want to show that we can assume, without loss of generality, that $w$ contains exactly one pure cyclic block.

\par Suppose that $w$ is primitive and contains more than one pure cyclic block. There can only be finitely many blocks since $G$ is finite; pick one and call it $B$. Thus, $B$ can be written as $(e_1, e_2, \ldots, e_n)$, where we assume that all of the $e_j$ belong to $w^-$. Then $w$ must be of the form $(w_1, e_1, w_2, e_2, \ldots, w_n, e_n),$ where each $w_j$ is an odd subwalk of $w$ that starts and ends at vertex $j$, as in Figure \ref{fig:primwithcycblock}.
\begin{figure}[ht]
  \centering
  \begin{tikzpicture}[line cap=round,line join=round,x=2cm,y=2cm]
    \coordinate (1) at (0,.75);
    \coordinate (2) at (-.75,0);
    \coordinate (3) at (-.45,-.75);
    \coordinate (4) at (.45,-.75);
    \coordinate (5) at (.75,0);

    \draw (1) node [anchor=north]{$1$};
    \draw (2) node [anchor=west]{$2$};
    \draw (3) node [anchor=south west]{$3$};
    \draw (4) node [anchor=south east]{$n-1$};
    \draw (5) node [anchor=east]{$n$};

    \draw (1) edge node [anchor=south east]{$e_1$} (2);
    \draw (2) edge node [anchor=east]{$e_2$} (3);
    \draw (3) edge [dotted] node [anchor=north]{} (4);
    \draw (4) edge node [anchor=west]{$e_{n-1}$} (5);
    \draw (5) edge node [anchor=south west]{$e_n$} (1);
    \draw (1) edge [loop, distance=2cm] node [anchor=south]{$w_1$} (1);
    \draw (2) edge [loop, out=120, in=210, distance=2cm] node [anchor=south east]{$w_2$} (2);
    \draw (3) edge [loop, out=180, in=270, distance=2cm] node [anchor=north]{$w_3$} (3);
    \draw (4) edge [loop, out=270, in=0, distance=2cm] node [anchor=north]{$w_{n-1}$} (4);
    \draw (5) edge [loop, out=330, in=60, distance=2cm] node [anchor=south west]{$w_n$} (5);

    \draw [fill=black] (1) circle (0.03);
    \draw [fill=black] (2) circle (0.03);
    \draw [fill=black] (3) circle (0.03);
    \draw [fill=black] (4) circle (0.03);
    \draw [fill=black] (5) circle (0.03);
  \end{tikzpicture}
  \caption{A primitive walk $w$ with cyclic block $(e_1,e_2,\ldots,e_n)$ where $n \ge 4$.}
  \label{fig:primwithcycblock}
\end{figure}

Now, let $w'$ be the walk $w' = (w_1, e_1, ..., w_{n-2}, e_{n-2}, e_{n-1}, e_n)$ that follows the same path as $w$, only skipping over the last two odd walks $w_{n-1}$ and $w_n$. Then, $B$ is still a cyclic block of $w'$, but it is not a pure cyclic block since the edge $e_{n-1}$  belongs to $w^+$ instead of $w^-$. This also means that $w'$ has strictly fewer pure cyclic blocks than $w$ does. It is possible that $w'$ now has no pure cyclic blocks, if all of the other pure cyclic blocks of $w$ were contained in $w_{n-1}$ or $w_n$. In this case, we pick a different starting vertex so that at least some of the remaining cyclic blocks of $w$ are not in $w_{n-1}$ or $w_n$. Therefore, $w'$ has at least one pure cyclic block, since we assumed that $w$ had more than one pure cyclic block.

\par We check that $w'$ is still primitive in this construction. By Proposition \ref{primitivebinomial}, we need to check three conditions. Since $w$ is primitive, every block of $w$ is cyclic or a cut edge. Thus, since $w'$ is a subwalk of $w$, all blocks of $w'$ are blocks of $w$, meaning the first condition is satisfied. Similarly, any multiple edge of $w'$ is also a multiple edge of $w$, so the edges of $w'$ must be cut edges and traversed exactly twice. Finally, as a result of the construction of $w'$, no additional cut verticies were made. In particular, the set of cut vertices of $w'$ is exactly the set of cut vertices of $w$, excluding the cut vertices in $w_{n-1}, w_n$ and the vertices that connect these two subwalks to $B$. All three conditions are therefore satisfied, so $w'$ is primitive.

\par If $w'$ has more than one pure cyclic block, we can repeat this construction on another pure cyclic block of $w'$ to get another primitive subwalk that has strictly fewer pure blocks than $w'$ does, but that has at least one. We can repeat this process until it terminates at a walk with exactly one pure cyclic block. Call this walk $w$.

\par Since $w$ has one pure cyclic block $B$, the binomial $b_w$ is not contained in $U_G$ but is primitive. Suppose that $B$ has at least $4$ edges, so $n \geq 4$. Now, repeat the construction above to get a subwalk $w'$ of $w$ with no pure cyclic blocks. By Corollary \ref{universalgroebnerbasiselt}, this means that $w'$ is contained in the universal Gr\"obner basis of $I_G$. However, the edges $e_1, e_2, \ldots, e_{n-2}, e_n$ are all contained in $w^-$, which means that the vertices $1, 2, \ldots, n-2$ are all sinks of the block $B$ of $w'$. Since $n \geq 4$, vertices $1$ and $2$ are both sinks, and they have distance one since they are connected by the edge $e_1$, so that $w'$ is not strongly primitive. By the result of \cite{RTT12} discussed above, this implies that $w'$ is not minimal. Since we have an element of $\mathcal{U}_G$ that is not minimal, thus not contained in a minimal set of generators, it must be the case that $\mathcal{U}_G$ is not a minimal generating set, so $G$ is not robust.

\par Finally, we consider the special case where $n=3$. That is, the single pure cyclic block $B$ of $w$ has only three edges $e_1, e_2, e_3$, as in Figure \ref{fig:primwithcycblock3}.

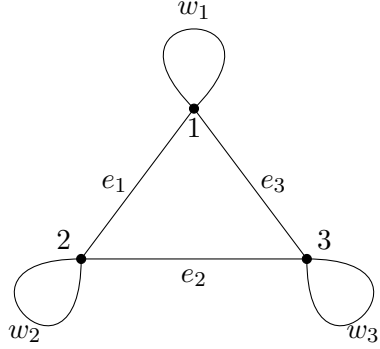
\begin{figure}[ht]
  \centering
  \begin{tikzpicture}[line cap=round,line join=round,x=2cm,y=2cm]
    \coordinate (1) at (0,1);
    \coordinate (2) at (-.75,0);
    \coordinate (3) at (.75,0);

    \draw (1) node [anchor=north]{$1$};
    \draw (2) node [anchor=south east]{$2$};
    \draw (3) node [anchor=south west]{$3$};

    \draw (1) edge node [anchor=east]{$e_1$} (2);
    \draw (2) edge node [anchor=north]{$e_2$} (3);
    \draw (3) edge node [anchor=west]{$e_3$} (1);
    \draw (1) edge [loop, distance=2cm] node [anchor=south]{$w_1$} (1);
    \draw (2) edge [loop, out=180, in=270, distance=2cm] node [anchor=north]{$w_2$} (2);
    \draw (3) edge [loop, out=270, in=0, distance=2cm] node [anchor=north]{$w_3$} (3);

    \draw [fill=black] (1) circle (0.03);
    \draw [fill=black] (2) circle (0.03);
    \draw [fill=black] (3) circle (0.03);
  \end{tikzpicture}
  \caption{A primitive walk $w$ with cyclic block $(e_1,e_2,\ldots,e_n)$ where $n = 3$.}\label{fig:primwithcycblock3}
\end{figure}

Now let $w'$ be the primitive walk obtained by the construction above, that is, $w' = (w_1, e_1, e_2, e_3)$. Let $w''$ be the closed even walk $w'' = (w_1, e_1, e_2, w_3, e_2, e_1)$. Neither $w'$ nor $w''$ have any pure cyclic blocks, since $B$ was the only pure cyclic block of $w$ by assumption, so as long as they are primitive, their corresponding binomials will be elements of $\mathcal{U}_G$. By the above construction, $w'$ is primitive, and it is easy to see that $w''$ is as well, using the fact that $w$ is primitive. Then, the binomial corresponding to the walk $w''$ is $b_{w''} = w_1^+e_2^2w_3^- - w_1^-e_1^2w_3^+$, where $w_j^+$ is the odd part of $w_j$, and $w_j^-$ is the even part of $w_j$. Similarly, the binomial corresponding to $w'$ is $b_{w'} = w_1^+e_2 - w_1^-e_1e_3$. By the above argument, both of these are elements of $\mathcal{U}_G$. However, one term of $b_{w'}$ divides a term of $b_{w''}$. By Corollary \ref{ugbnondividing}, this implies that $\mathcal{U}_G$ is not robust.
\end{proof}

The following proposition gives an application of Theorem \ref{urobustiffgrobust}.  It describes one modification to any graph G that preserves robustness.  Example \ref{amycounter} shows that modifying nonrobust graphs can often have unpredictable effects on $\mathcal{U}_G$.

 \begin{prop}\label{exchangeoneedgeforthree}
Let $G$ be a graph and $b$ be an edge.  Let $G'$ be the graph obtained from $G$ by replacing $b$ with three edges.  Then $|\mathcal{G}_G|=|\mathcal{G}_{G'}|$.  If $G$ is robust, then so is $G'$. 
 \end{prop}

\begin{figure}[ht]
  \centering
  \begin{tikzpicture}[line cap=round,line join=round,x=1.5cm,y=1.5cm]
    \node (a) at (0,0) 
    {\begin{tikzpicture}
      \coordinate (v1) at (0,0);
      \coordinate (v2) at (1,0);

    \draw (v1) edge node [anchor=south]{$b$} (v2);
      
      \draw[fill=black] (v1) circle (0.04);
      \draw[fill=black] (v2) circle (0.04);

    \end{tikzpicture}};
    \node (b) at (a.east)[anchor=west,xshift=0.5cm]{$\mapsto$};
    \node (c) at (b.east)[anchor=west,xshift=0.5cm] 
    {\begin{tikzpicture}
      \coordinate (v1) at (0,0);
      \coordinate (v2) at (.25,.3);
      \coordinate (v3) at (.75,.3);
      \coordinate (v4) at (1,0);

    \draw (v1) edge node [anchor=south east]{$a$} (v2);
    \draw (v2) edge node [anchor=south]{$b'$} (v3);
    \draw (v3) edge node [anchor=south west]{$c$} (v4);
      
      \draw[fill=black] (v1) circle (0.04);
      \draw[fill=black] (v2) circle (0.04);
      \draw[fill=black] (v3) circle (0.04);
      \draw[fill=black] (v4) circle (0.04);
    \end{tikzpicture}};

    \draw (a) node[anchor=north,yshift=-0.8cm]{};
    \draw (b) node[anchor=north,yshift=-0.8cm]{};
    \draw (c) node[anchor=north,yshift=-0.8cm]{};
  \end{tikzpicture}
  \caption{Construction in Proposition \ref{exchangeoneedgeforthree}.}
  \label{fig:exchangeoneedgeforthree}
\end{figure}
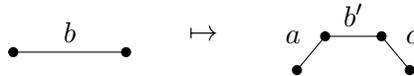

\begin{proof}
  Consider when we replace the edge $b \in E(G)$ with $\{a,b',c\}$ as in Figure \ref{fig:exchangeoneedgeforthree}, producing a new graph $G'$. Notice that any walk that contains one of $a,b',c$ must contain them all.  Let $\mathcal{L}_G$ denote the set of walks on $G$.   Define the map $\varphi: \mathcal{L}_G \to \mathcal{L}_{G'}$ which takes a walk $w$ to its image in $G$ replacing all instances of $b$ with $\{a,b',c\}$.  On binomials, 
  $$\varphi(mb^\ell - n) = m(ac)^\ell - n(b')^\ell$$
where $m, n$ are monomials not involving $b$.  It is straightforward to check that $\varphi$ provides a bijection between the primitive walks of $G$ and $G'$, proving that $|\mathcal{G}_G|=|\mathcal{G}_{G'}|$.  

Suppose that $G$ is robust.  We will show that $\mathcal{G}_{G'}$ is a minimal generating set for $I_{G'}$.  Let $\mathcal{G}_{G} = \{w_1, \ldots, w_k\}$ and  $\mathcal{G}_{G'} = \{w'_1, \ldots, w'_k\}$ where $\varphi(w_j) = w'_j$.  If the binomials $w'_j$ do not minimally generate $I_{G'}$ then one, say, $w'_1$ must be an polynomial linear combination of the others.  But then this must mean that one term of $w'_1$ is divisible by a term of another $w'_i$, say $w'_2$.  But then it follows that one term of $w_1$ is divisible by a term of $w_2$.  But by Theorem \ref{urobustiffgrobust} we have that $\mathcal{U}_G = \mathcal{G}_G$ and this is a contradiction by Lemma \ref{ugbnondividing}.

\end{proof}

\begin{example}\label{amycounter}
Notice that for non-robust graphs, the number of minimal generators and the set $\mathcal{U}_G$ are very senstive to changes in the graph $G$.  For example, consider the graphs $G$ and $G'$ in Figure \ref{fig:essentialrobust}.
These graphs are not robust.  The left graph has, $\mu(I_G) = 3$ while the graph on the right has $\mu(I_{G'})=4$.  The walks $w$ and $w'$ that traverse each edge in $G$ and $G'$ once are both primitive, but $w\notin \mathcal{U}_G$ whereas $w'\in \mathcal{U}_{G'}$.

\begin{figure}[ht]
  \centering
  \begin{tikzpicture}[line cap=round,line join=round,x=1.5cm,y=1.5cm]
    \node (a) at (0,0){
   \begin{tikzpicture}[line cap=round,line join=round,x=0.5cm,y=0.5cm]
      \clip (-4.5,-3) rectangle (4.5,3.5);
      \coordinate (1) at (-4.05,2.34);
      \coordinate (9) at (-2.7,0);
      \coordinate (2) at (-1.35,2.34);
      \coordinate (3) at (1.35,2.34);
      \coordinate (4) at (4.05,2.34);
      \coordinate (5) at (2.7,0);
      \coordinate (6) at (0,0);
      \coordinate (8) at (-1.35,-2.34);
      \coordinate (7) at (1.35,-2.34);

      \draw (1) edge (2);
      \draw (2) edge (3);
      \draw (3) edge (4);
      \draw (4) edge (5);
      \draw (3) edge (5);
      \draw (3) edge (6);
      \draw (6) edge (7);
      \draw (7) edge (8);
      \draw (6) edge (8);
      \draw (2) edge (6);
      \draw (2) edge (9);
      \draw (1) edge (9);
      
    \draw (2) edge node [anchor=south]{$+$} (3);
    \draw (3) edge node [anchor=west]{$+$} (6);
    \draw (2) edge node [anchor=east]{$+$} (6);

      \draw[fill=black] (1) circle (0.12);
      \draw[fill=black] (2) circle (0.12);
      \draw[fill=black] (3) circle (0.12);
      \draw[fill=black] (4) circle (0.12);
      \draw[fill=black] (5) circle (0.12);
      \draw[fill=black] (6) circle (0.12);
      \draw[fill=black] (7) circle (0.12);
      \draw[fill=black] (8) circle (0.12);
      \draw[fill=black] (9) circle (0.12);
    \end{tikzpicture}};
    \node (b) at (a.east)[anchor=west,xshift=0.5cm]{$\mapsto$};
    \node (c) at (b.east)[anchor=west,xshift=0.5cm]{

       \begin{tikzpicture}[line cap=round,line join=round,x=0.5cm,y=0.5cm]
      \clip (-4.5,-3) rectangle (4.5,3.9);
      \coordinate (1) at (-4.05,2.34);
      \coordinate (9) at (-2.7,0);
      \coordinate (2) at (-1.35,2.34);
      \coordinate (3) at (1.35,2.34);
      \coordinate (4) at (4.05,2.34);
      \coordinate (5) at (2.7,0);
      \coordinate (6) at (0,0);
      \coordinate (8) at (-1.35,-2.34);
      \coordinate (7) at (1.35,-2.34);
      \coordinate (10) at (-.6,3.05);
      \coordinate (11) at (.6,3.05);

      \draw (1) edge (2);
      \draw (2) edge (10);
      \draw (11) edge (10);
      \draw (3) edge (11);
      \draw (3) edge (4);
      \draw (4) edge (5);
      \draw (3) edge (5);
      \draw (3) edge (6);
      \draw (6) edge (7);
      \draw (7) edge (8);
      \draw (6) edge (8);
      \draw (2) edge (6);
      \draw (2) edge (9);
      \draw (1) edge (9);

    \draw (2) edge node [anchor=south east]{$+$} (10);
    \draw (3) edge node [anchor=west]{$+$} (6);
    \draw (2) edge node [anchor=east]{$+$} (6);
    \draw (10) edge node [anchor=south]{$-$} (11);
    \draw (3) edge node [anchor=south west]{$+$} (11);

      \draw[fill=black] (1) circle (0.12);
      \draw[fill=black] (2) circle (0.12);
      \draw[fill=black] (3) circle (0.12);
      \draw[fill=black] (4) circle (0.12);
      \draw[fill=black] (5) circle (0.12);
      \draw[fill=black] (6) circle (0.12);
      \draw[fill=black] (7) circle (0.12);
      \draw[fill=black] (8) circle (0.12);
      \draw[fill=black] (9) circle (0.12);
      \draw[fill=black] (10) circle (0.12);
      \draw[fill=black] (11) circle (0.12);
    \end{tikzpicture}};

    \draw (a) node[anchor=north,yshift=-0.8cm]{};
    \draw (b) node[anchor=north,yshift=-0.8cm]{};
    \draw (c) node[anchor=north,yshift=-0.8cm]{};
  \end{tikzpicture}
  \caption{Example where $\lvert\mathcal{M}_G\rvert < \lvert\mathcal{M}_{G'}\rvert$.}
  \label{fig:essentialrobust}
\end{figure}
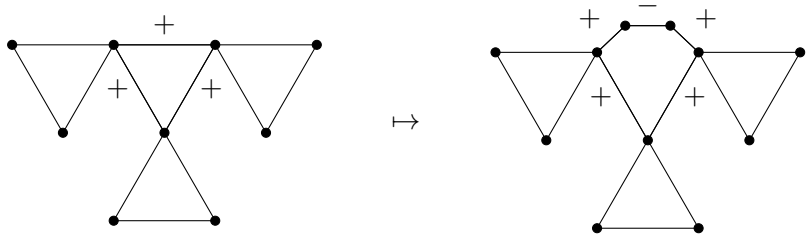

\end{example}

\begin{example}\label{amycounter2}
The reverse implication in the Proposition is false as shown in Figure \ref{noreverseamy}.  The graph on the right is obtained by contracting three edges into one. The graph on the left is robust, but the one of the right is not.

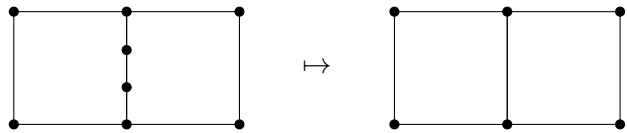
\begin{figure}[ht]
  \centering
  \begin{tikzpicture}[line cap=round,line join=round,x=1.5cm,y=1.5cm]
    \node (a) at (0,0) 
    {\begin{tikzpicture}
      \coordinate (v1) at (0,0);
      \coordinate (v2) at (0,1);
      \coordinate (v3) at (1,1);
      \coordinate (v4) at (1,0);
      \coordinate (v5) at (2,0);
      \coordinate (v6) at (2,1);
      \coordinate (v7) at (1,0.66);
      \coordinate (v8) at (1,0.33);
      
      \draw (v1) -- (v2) -- (v3) -- (v4) -- (v1);
      \draw (v5) -- (v6) -- (v3) -- (v4) -- (v5);
      
      \draw[fill=black] (v1) circle (0.04);
      \draw[fill=black] (v2) circle (0.04);
      \draw[fill=black] (v3) circle (0.04);
      \draw[fill=black] (v4) circle (0.04);
      \draw[fill=black] (v5) circle (0.04);
      \draw[fill=black] (v6) circle (0.04);
      \draw[fill=black] (v7) circle (0.04);
      \draw[fill=black] (v8) circle (0.04);
    \end{tikzpicture}};
    \node (b) at (a.east)[anchor=west,xshift=0.5cm]{$\mapsto$};
    \node (c) at (b.east)[anchor=west,xshift=0.5cm] 
    {\begin{tikzpicture}
      \coordinate (v1) at (0,0);
      \coordinate (v2) at (0,1);
      \coordinate (v3) at (1,1);
      \coordinate (v4) at (1,0);
      \coordinate (v5) at (2,0);
      \coordinate (v6) at (2,1);
      
      \draw (v1) -- (v2) -- (v3) -- (v4) -- (v1);
      \draw (v5) -- (v6) -- (v3) -- (v4) -- (v5);
      
      \draw[fill=black] (v1) circle (0.04);
      \draw[fill=black] (v2) circle (0.04);
      \draw[fill=black] (v3) circle (0.04);
      \draw[fill=black] (v4) circle (0.04);
      \draw[fill=black] (v5) circle (0.04);
      \draw[fill=black] (v6) circle (0.04);
    \end{tikzpicture}};

    \draw (a) node[anchor=north,yshift=-0.8cm]{};
    \draw (b) node[anchor=north,yshift=-0.8cm]{};
    \draw (c) node[anchor=north,yshift=-0.8cm]{};
  \end{tikzpicture}
  \caption{Counterexample to the reverse implication.}\label{noreverseamy}
\end{figure}
\end{example}


\section{Characterization of Robust Graph Ideals}

We begin this section with a definition and characterization of indispensable walks:
\begin{definition}
A primitive walk $w$ of a graph $G$ is \emph{indispensible} if the corresponding binomial $B_w$ or its negation appears in every minimal generating set of $I_G$.
\end{definition}

\begin{prop}[{\cite[Thm.~4.14]{RTT12}}]\label{indispensable}
  A primitive walk $w$ is indispensable if and only if
  \begin{description}\itemsep0pt\parskip0pt\parsep0pt
    \item[I1] $w$ has no even chords,
    \item[I2] $w$ has no bridges,
    \item[I3] $w$ has no effective crossings, and
    \item[I4] $w$ is strongly primitive.
  \end{description}
\end{prop}

\begin{lemma}\label{robustiffprimindisp}
  $I_G$ is robust if and only if all primitive elements are indispensable. 
\end{lemma}
\begin{proof}
  Since $\mathcal{U}$ is a generating set, we have the inclusions
  \begin{equation*}
    \mathcal{I} \subset \mathcal{U} \subset \mathcal{G},
  \end{equation*}
where $\mathcal{I}$ denotes the indispensable elements. Notice that by definition, $\mathcal{I}$ minimally generates an ideal.  Now if $I_G$ is robust, then by Theorem \ref{urobustiffgrobust}, we have equalities everywhere, since in particular our minimal generating set is unique by the uniqueness of $\mathcal{U}$ and so $\mathcal{I} = \mathcal{U}$. Conversely, if $\mathcal{G} = \mathcal{I}$, $\mathcal{U} = \mathcal{I}$ and so we have robustness.
\end{proof}

\begin{lemma}[{\cite[Cor.~3.3]{RTT12}}]\label{primcutvert}
  A cut vertex in the induced subgraph $W$ of a primitive walk $w$ separates the graph in two vertex-disjoint parts, the total number of edges of the cyclic blocks in each part is odd.
\end{lemma} 
\begin{lemma}\label{findpathlem}
  Suppose $v_0$ is a cut vertex in the induced subgraph $W$ of a primitive walk $w$. Then, there exist two simple paths $p_1,p_2$ from $v_0$ to two odd cyclic blocks $B_1,B_2$ of $W$ respectively, where $B_1,B_2$ are in the two different parts of $W$ in Lemma \ref{primcutvert}.
\end{lemma} 
\begin{proof}
Let $w$ be a primitive walk with cut vertex $v_0$.  By Lemma \ref{primcutvert}, $v_0$ separates $W$ into two connected subgraphs $W_1,W_2$ both of which have at least one cyclic block. Further, $W_1,W_2$ are vertex disjoint except for $v_0$.  Now since the $W_i$ are connected, we can find paths joining $v_0$ to odd cycles in each $W_i$.  These can be chosen to be simple by omitting any ``loops''.
\end{proof}
\begin{lemma}\label{simplepathlem}   If $v_i,v_j$ are distinct vertices in a walk $w$ then there is a simple path $p$ contained in $w$ that connects $v_i,v_j$.
\end{lemma}
\begin{proof}
If $v_i, v_j$ are both vertices in $w$ then they are clearly connected via a subwalk of $w$.  If this subwalk is not a simple path, then this must be because the walk doubles back upon itself at some point.  Removing all such ``loops'' yields a simple path connecting $v_i, v_j$. 

\end{proof}

To determine questions about robustness, we can simplify the characterization of indispensible elements given above. To make this clear, we offer the following proposition.
\begin{prop}\label{dontneedstrongprimitivity}
  $I_G$ is robust if and only if all primitive elements satify conditions I1, I2, and I3 of Proposition \ref{indispensable}.
\end{prop}
\begin{proof}
If $I_G$ is robust, then all primitive elements are indispensible by Lemma \ref{robustiffprimindisp}, so they satisfy conditions I1 through I4. In particular, they satisfy the first three. Now suppose that all primitive elements of $I_G$ satisfy conditions I1, I2, and I3, but suppose for the sake of a contradiction that there exists a primitive walk $w$ that does not satisfy I4, that is, it is not strongly primitive.

The walk $w$ must be of type P1, P2, or P3.  

P1: An even cycle is biconnected and its underlying graph is $2$-regular thus no sinks exist so $w$ is strongly primitive.  

P2: For two odd cycles joined at a cut vertex, the cut vertex is the unique sink, so $w$ is strongly primitive. 

P3: Suppose $w$ consists of two odd cycles joined along paths.  Let $e$ denote an edge connecting two sinks $s_1, s_2$ in the same cyclic block $B$. Applying Lemma \ref{findpathlem} to $s_1$ on the part of the graph not containing $B$ we get from Lemma \ref{primcutvert}, we get an odd cycle $c_1$ and a simple path $p_1$ connecting $c_1$ to $s_1$. By the same argument, we get an odd cycle $c_2$ and a simple path $p_2$ connecting $c_2$ to $s_2$. Now let $q$ denote the path in $B$ connecting $s_1$ and $s_2$ that does not contain $e$. Then, $e$ forms a bridge of the C3 circuit $c := (c_1,p_1,q,p_2^{-1},c_2,p_2,q^{-1},p_1^{-1}),$ contradicting I2.
\end{proof}
We are now ready to state our classification of robust graph ideals.
\begin{thm}\label{charcircuits}
  $I_G$ is robust if and only if the following conditions are satisfied.
  \begin{description}\itemsep0pt\parskip0pt\parsep0pt
    \item[R1] No circuit of $G$ has an even chord,
    \item[R2] No circuit of $G$ has a bridge,
    \item[R3] No circuit of $G$ contains an effective crossing, and
    \item[R4] No circuit of $G$ shares exactly one edge (and no other vertices) with another circuit such that the shared edge is part of a cyclic block in both circuits.
  \end{description}
\end{thm}

In particular, this implies that questions of robustness can be answered by looking at the circuits as they lie on the graph.

\begin{proof}
By Proposition \ref{dontneedstrongprimitivity}, it suffices to show that all primitive walks satisfy conditions I1, I2, and I3 if and only if $G$ satisfies conditions R1 through R4. To do this, we'll show the contrapositive statement, which is that one of R1 through R4 is not satisfied if and only if there exists a primitive walk that doesn't satisfy one of I1 though I3. We begin with the forward direction of this new statement.
\vspace{1em}
\par \underline{$\Rightarrow : \neg$ R1, R2, R3} Suppose that one of R1 through R3 is false. Then, there exists a circuit $w$ with either an even chord, a bridge, or an effective crossing. But all circuits are primitive, so $w$ is a primitive walk that doesn't satisfy one of I1 through I3.
\par \underline{$\neg$ R4} Suppose that R4 is false, then there are two circuits $c$ and $c'$ that share exactly one edge $e$ (and no other vertices), where $e$ belongs to a cyclic block of both $c$ and $c'$. First consider the case where $c$ and $c'$ are C1 circuits. Then write $c = (e,w)$ where $w$ is an odd simple walk connecting the vertices of $e$, and similarly put $c' = (e,w')$. Since $c$ and $c'$ share no vertices other than the two of $e$, the walks $w$ and $w'$ share only these two vertices as well. Consider the new walk $u = (w',w^{-1})$ that starts and ends at one vertex of $e$. The walk $u$ is an even cycle since $w$ and $w'$ are both odd only share the 2 vertices that connect them. However, $e$ is an even chord of $u$, by construction. Therefore, $u$ is primitive and doesn't satisfy I1.
\par Now consider the case where $c = (e,w)$ is a C1 circuit and $c' = (c'_1, w', c'_2, w'^{-1})$ is a circuit of type C2 or C3. (In the case where it is C2, $w'$ is the empty walk). We can suppose without loss of generality that the edge $e$ occurs in $c'_2$, so put $c'_2 = (u'_1, e, u'_2)$, and denote by $v'$ the vertex that connects $c'_2$ to $w'$. Consider the new walk $u = (c'_1, w', u'_1, w, u'_2, w'^{-1})$. It is a C3 (or C2, if $w'$ is empty) walk, by the disjointness condition on $c$ and $c'$. If both $u'_1$ and $u'_2$ have positive length, then $e$ is an even chord of $u$ connecting two vertices in $c'_2$, neither of which are $e$, by construction. Thus, $u$ doesn't satisfy condition I1. If one of $u'_1$, $u'_2$ is empty (they can't both be empty since $c'_2$ is not a loop), then $e$ is a chord connecting $v'$ to another vertex in $c'_2$, so that $e$ is a bridge of $u$. In this case, $u$ doesn't satisfy condition I2.
\par Finally we can consider the case where $c = (c_1, w, c_2, w^{-1})$ and $c' = (c'_1, w', c'_2, w'^{-1})$ are both C2 or C3 circuits. Without loss of generality, suppose that $e$ is contained in both $c_2$ and $c'_2$. As above, write $c_2 = (u_1, e, u_2)$ and $c'_2 = (u'_1, e, u'_2)$. We can pick the order of $u'_1$ and $u'_2$ so that the last vertex of $u_1$ is the first vertex of $u'_2$. Now consider the closed even walk $u = (c_1, w, u_1, u'_2, w', c'_1, w'^{-1}, u'_1, u_2, w^{-1})$. The walk $u$ is primitive of type P3 by the disjointness condition on $c$ and $c'$. Let $v$ be the vertex connecting $c_2$ to $w$, and define $v'$ similarly for $c'$. If none of $u_1, u_2, u'_1, u'_2$ are empty, then $e$ is a chord of $u$ connecting two vertices in the cyclic block $(u_1, u'_2, u'_1, u_2)$, neither of which are $v$ or $v'$. In this case, $e$ is an even chord by construction, so that $u$ doesn't satisfy I1. Alternatively, if one of these four walks in empty, then at least one vertex of $e$ will be $v$ or $v$, so that $e$ is a bridge of $u$, so that $u$ doesn't satisfy I2.
\vspace{1em}
Now we have shown that if one of R1 through R4 is false, then there exists a primitive walk that doesn't satisfy at least one of I1 through I3. 
\par \underline{$\Leftarrow :$} For the other direction, let $w$ be a primitive walk that doesn't satisfy at least one of conditions I1, I2, and I3. First suppose $w$ is of type P1 or P2. Since P1 and P2 primitive walks are exactly C1 and C2 circuits, and since $w$ has either an even chord, a bridge, or an effective crossing, we have shown that one of R1 through R3 must be false for our graph $G$.
\vspace{1em}
\par \underline{$\neg$ I1} Now suppose $w$ is a primitive walk of type $P3$. First we consider the case where $w$ doesn't satisfy the condition I1, so that $w$ has an even chord $f$ in one of its cyclic blocks $B$. First, suppose that $B$ is odd. If we look at $B = (e_1, \ldots, e_n)$ as an odd cycle, then $f$ splits $B$ into one side with an odd number of edges $(e_1,e_2,\ldots,e_k)$ where $k$ is odd. Denote by $c'$ the even cyclic block $(e_1,\ldots,e_k,f)$, and denote by $c_1$ the odd cyclic block $(e_{k+1},\ldots,e_n,f)$. Since $f$ is an even chord, $c_1$ is a cyclic block of some even walk. This means that one of the vertices of $c_1$ that is not a vertex of $f$ must be the start and end of some odd path that is vertex-disjoint from the rest of the the subgraph induced by $w \setminus B.$ 
Because $w$ is primitive, this vertex $v$ must be a cut vertex of $w$. Applying Lemma \ref{findpathlem} to the connected component of $w\setminus{v}$ that does not contain $f$, we get a simple path $p$, potentially empty, from $c_1$ to an odd cycle $c_2$. Then, $c = (c_1, p, c_2, p^{-1})$ is a C3 (or C2 if $p$ is empty) circuit. By construction, $c$ and $c'$ are two circuits that share exactly one edge (and no other vertices) contained in a cyclic block of both of them. This contradicts R4.
\par The second subcase is when $B$, the cyclic block containing the chord $f$, is even. If $f$ is an even chord of $B$ when $B = (e_1, \ldots, e_n)$ is considered as a closed even cycle, then we have shown that R1 is not true, with $B$ being the offending circuit. Suppose that $f$ is an odd chord of $B$ when considered as a closed even cycle. Then $f$ divides $B$ into two odd cycles $c_1 = (e_1,\ldots,e_k,f)$ and $c'_1 = (e_{k+1},\ldots,e_n,f)$. As above, one of the vertices $v$ of $c_1$ that is not a vertex of $f$ must be the start and end of some odd path. Again, we apply Lemma \ref{findpathlem} to the part of $w$ that doesn't contain $f$, using $v$ as the cut vertex to get a simple path $p$ and odd cycle $c_2$ that make $c = (c_1, p, c_2, p^{-1})$ into a C3 or C2 circuit. Symmetrically, we can do the same for $c'_1$ to get another C3 or C2 circuit $c' = c'_1, p', c'_2, p'^{-1})$. By construction, $c$ and $c'$ share exactly one edge $f$ and no other vertices, and $f$ is contained in a cyclic block of both. This shows that R4 is false for our graph.
\vspace{1em}
\par \underline{$\neg$ I2}
Suppose $w$ is a P3 walk with induced subgraph $W$ which has a bridge, so that it does not satisfy I2. Then, $w$ has a bridge $f$ connecting two vertices $v_i,v_j$ that lie in different blocks $B_1$ and $B_2$, which share at most one vertex by the definition of block, for otherwise $B_1 \cup B_2$ would be a larger biconnected subgraph of $W$, contradicting the definition of block. Note that by Proposition \ref{primitivebinomial}, $v_i$ is either a cut vertex, or is a non-cut vertex in a cycle---in the latter case, the cycle can be even or odd; the same applies to $v_j$. We claim there is a C2 or C3 circuit such that $f$ is a bridge. Our strategy is to connect $v_i,v_j$ with a simple path $p_0$, and then find odd cycles on both ends of $p$.
\par We first find a simple path $p_0$ between $v_i,v_j$. Let $\mathcal{P}$ be the set of simple paths between $v_i,v_j$ that are contained in $W$; it is non-empty by Lemma \ref{simplepathlem}. We then let $p_0$ be the simple path with the minimal number of cut vertices of $W$ (not necessarily unique), which exists since the number of cut vertices for any walk in $W$ is a well-defined natural number.
\par Suppose $v_i$ is a cut vertex. Then, applying Lemma \ref{findpathlem} to the connected component of $W$ not containing $v_j$ gives a path $p_1$ from $v_i$ to some odd cycle $c_1$. Thus, the path $p = (p_0,p_1)$ goes from $v_j$, through $v_i$, and ends at $c_1$.
\par Suppose $v_i$ is a non-cut vertex contained in an even cycle $B_1$. We claim that $B_1$ contains a cut vertex of $W$ that is not in $p_0$. Suppose $B_1$ contains the cut vertices $x_1,x_2,\ldots,x_n$, appearing in that order such that tracing the path around $B_1$ goes through $x_1$, then $x_2$, etc., until it goes through $x_n$, and then back to $x_1$, and such that $v_i$ appears between $x_n$ and $x_1$. Now if $p_0$ goes through every cut vertex, then after possible relabeling $p_0$ goes through $x_1$, $x_2$, etc., until it goes through $x_n$ and then out of $B_1$. But then, the path connecting $v_i$ to $x_n$ and then out of $B_1$ is a simple path connecting $v_i,v_j$ with fewer cut vertices, contradicting the construction of $p$. Thus, letting $x_0$ denote the cut vertex in $B_1$ not contained in $p$, there exists a simple path $p_1$ connecting $x_0$ to $v_i$ contained in $B_1$, for otherwise $p$ would not be simple. Now applying Lemma \ref{findpathlem} at $x_0$ as in the case when $v_i$ is a cut vertex gives a path $p_2$ from $x_0$ to some odd cycle $c_1$, and the path $p = (p_0,p_1,p_2)$ goes from $v_j$, through $v_i$, and ends at $c_1$.
\par Now suppose $v_i$ is a non-cut vertex contained in an odd cycle $B_1$. We claim we can find a subpath $p$ of $p_0$ such that it only intersects $B_1$ once. Orienting $p_0$ such that it starts at $v_i$, we can find the last vertex $v_1$ such that $v_1 \in B_1$. Letting $p_1$ be the subpath of $p$ starting at $v_1$, we have the subpath desired. Letting $p_2$ be the simple path from $v_1$ to $v_i$ fully contained in $B_1$, we see that the path $p = (p_1,p_2)$ goes from $v_j$ to an odd cycle $B_1$ which contains $v_i$.
\par Repeating the same process at $v_j$, we see that we can thereby construct a C2 or C3 circuit such that $f$ is a bridge between $v_i,v_j$.
\vspace{1em}
\par \underline{$\neg$ I3} Finally, suppose $w$ is a P3 walk that violates I3, that is, $w$ has an effective crossing in some cyclic block $B = (e_1,\ldots,e_n)$ of $w$. Let $f = (v_1,v_3)$ and $f' = (v_2,v_4)$ be the two odd chords of $w$ that cross strongly effectively in $w$. These chords divide $B$ up into 4 segments of edges, $s_1 = (e_1,\ldots,e_k), s_2 = (e_{k+1},\ldots,e_{l}), s_3 = (e_{l+1},\ldots,e_{m}), s_4 = (e_{m+1},\ldots,e_n)$, where the vertex between $s_j$ and $s_{j+1}$ is $v_{j+1}$ and the vertex between $s_4$ and $s_1$ is $v_1$. For each $i$, if $s_i$ has an even number of edges, then at least one vertex $v$ in $s_i$ that is not one of the $v_i$ must be the start and end of an odd path that is vertex disjoint from the rest of $w$. This is the case because $f$ and $f'$ cross strongly effectively, so there must be an odd number of edges along $w$ between where one starts and the other ends.
\par If $B$ is an odd cyclic block, then the sum of the lengths of the $s_i$ is odd, so either three are odd and one is even or three of them are even and one is odd.  In the former case, suppose that $s_1$ is the even one and let $v$ be a vertex as described above. By Lemma \ref{findpathlem}, there is a path $p$ and an odd cycle $c$ contained in the part of $w$ that does not contain $B$. But then, $(B,p,c,p^{-1})$ is a closed even walk that is either a C2 or C3 circuit, and $f$ and $f'$ cross effectively in this walk, contradicting R3. In the latter case, suppose without loss of generality that $s_1, s_2, s_3$ are even and $s_4$ is odd. Let $v$ and $v'$ be vertices in $s_1$ and $s_2$, respectively, that are not one of the $v_i$, that are the starting and ending points of odd walks in the manner described above. As before, in each case, we can find paths $p$ and $p'$ connecting $v$ and $v'$ to odd cyclic blocks $c$ and $c'$. Let $q$ be the walk that goes along $B$ connecting $v$ to $v'$ that goes through $v_1$, $v_4$, and $v_3$. Then the path $(c,p^{-1},q,p',c',p'^{-1},q^{-1},p)$ is a C3 circuit that has $f$ as a bridge, contradicting condition R2.
\par Finally, consider the case where $B$ is an even cyclic block. The sum of the lengths of the $s_i$ is even, so they are either all odd, all even, or two are even and two are odd. If each of the $s_i$ is odd, then $B$ is a closed even cycle with an effective crossing, negating condition R3. If two are even and two are odd, then without loss of generality, we can say that $s_1$ is even and $s_2$ is odd. But then, the chord $f$ is an even chord of $B$ when considered as an even cyclic walk, contradicting R1. If all four are even, then we let $v$ and $v'$ be the vertices in $s_1$ and $s_2$, respectively, that are the start and end of a closed even walk. Proceeding as in the above case, we find a C3 circuit of the form $(c,p^{-1},q,p',c',p'^{-1},q^{-1},p)$ that has $f$ as a bridge, negating condition R2.
\end{proof}

\begin{remark}
The previous theorems suggest that questions of robustness for a graph $G$, which is naturally a question about the Universal Gr\"obner Basis, can be reduced to a question about the Graver Basis, and in turn, to one about the circuits of $G$. In light of this, it is a natural question to ask if it is true that $G$ is robust (that is, if the primitive walks of $G$ are precisely the indispensable walks of $G$) if and only if the circuits of $G$ are precisely the indispensable walks of $G$. It turns out that both directions are false, as demonstrated by the following graphs.
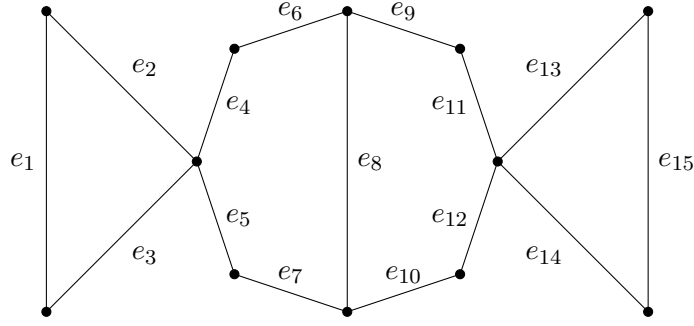
\begin{center}
\begin{tikzpicture}[line cap=round,line join=round,x=2cm,y=2cm]
    \coordinate (1) at (-2,1);
    \coordinate (2) at (-2,-1);
    \coordinate (3) at (-1,0);
    \coordinate (4) at (-.75,.75);
    \coordinate (5) at (-.75,-.75);
    \coordinate (6) at (0,1);
    \coordinate (7) at (0,-1);
    \coordinate (8) at (.75,.75);
    \coordinate (9) at (.75,-.75);
    \coordinate (10) at (1,0);
    \coordinate (11) at (2,1);
    \coordinate (12) at (2,-1);

    \draw (1) edge node [anchor=east]{$e_1$} (2);
    \draw (1) edge node [anchor=south west]{$e_2$} (3);
    \draw (2) edge node [anchor=north west]{$e_3$} (3);
    \draw (3) edge node [anchor=west]{$e_4$} (4);
    \draw (3) edge node [anchor=west]{$e_5$} (5);
    \draw (4) edge node [anchor=south]{$e_6$} (6);
    \draw (5) edge node [anchor=south]{$e_7$} (7);
    \draw (6) edge node [anchor=west]{$e_8$} (7);
    \draw (6) edge node [anchor=south]{$e_9$} (8);
    \draw (7) edge node [anchor=south]{$e_{10}$} (9);
    \draw (8) edge node [anchor=east]{$e_{11}$} (10);
    \draw (9) edge node [anchor=east]{$e_{12}$} (10);
    \draw (10) edge node [anchor=south east]{$e_{13}$} (11);
    \draw (10) edge node [anchor=north east]{$e_{14}$} (12);
    \draw (11) edge node [anchor=west]{$e_{15}$} (12);

    \draw [fill=black] (1) circle (0.03);
    \draw [fill=black] (2) circle (0.03);
    \draw [fill=black] (3) circle (0.03);
    \draw [fill=black] (4) circle (0.03);
    \draw [fill=black] (5) circle (0.03);
    \draw [fill=black] (6) circle (0.03);
    \draw [fill=black] (7) circle (0.03);
    \draw [fill=black] (8) circle (0.03);
    \draw [fill=black] (9) circle (0.03);
    \draw [fill=black] (10) circle (0.03);
    \draw [fill=black] (11) circle (0.03);
    \draw [fill=black] (12) circle (0.03);
\end{tikzpicture}
\end{center}
In this graph, the Graver basis consists of the following binomials:
\begin{align*}
  B_1 &= e_4e_7e_8e_{12}^2e_{15} - e_5e_6e_{10}^2e_{13}e_{14}, & B_7 &= e_1e_4^2e_8e_9e_{12} - e_2e_3e_6^2e_{10}e_{11},\\
  B_2 &= e_4e_7e_9e_{12} - e_5e_6e_{10}e_{11}, & B_8 &= e_1e_4^2e_9^2e_{13}e_{14} - e_2e_3e_6^2e_{11}^2e_{15},\\
  B_3 &= e_1e_5^2e_{10}^2e_{13}e_{14} - e_2e_3e_7^2e_{12}^2e_{15}, & B_9 &= e_1e_5^2e_8e_{10}e_{11} - e_2e_3e_7^2e_9e_{12},\\
  B_4 &= e_8e_{11}e_{12}e_{15} - e_9e_{10}e_{13}e_{14}, & B_{10} &= e_4e_7e_9^2e_{13}e_{14} - e_5e_6e_8e_{11}^2e_{15},\\
  B_5 &= e_1e_4e_5e_8 - e_2e_3e_6e_7, & B_{11} &= e_1e_5^2e_8^2e_{11}^2e_{15} - e_2e_3e_7^2e_9^2e_{13}e_{14},\\
  B_6 &= e_1e_4^2e_8^2e_{12}^2e_{15} - e_2e_3e_6^2e_{10}^2e_{13}e_{14}, & B_{12} &= e_1e_4e_5e_9e_{10}e_{13}e_{14} - e_2e_3e_6e_7e_{11}e_{12}e_{15}.
\end{align*}
\indent The circuits are precisely the first 11 binomials, as are the indispensible binomials. The binomial $B_{12}$ is primitive and an element of the Universal Gr\"obner Basis, but not indispensible. This is a counterexample to the backwards direction of the proposed statement above, since the set of circuits is the same as the set of indispensible walks, but the graph is not robust. This also gives an example of a graph that has a unique minimal generating set, but is not robust. To demonstrate a counterexample to the forward direction of the proposed statment, we offer the following robust graph.
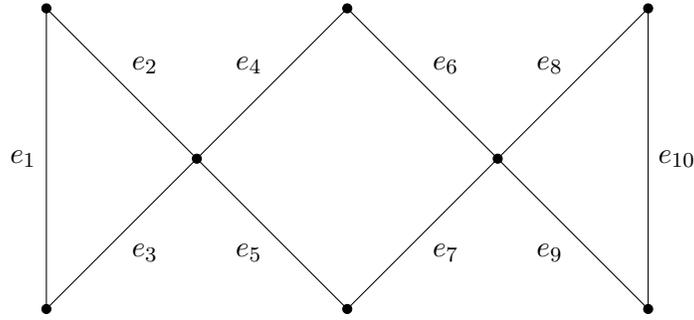
\begin{center}
\begin{tikzpicture}[line cap=round,line join=round,x=2cm,y=2cm]
    \coordinate (1) at (-2,1);
    \coordinate (2) at (-2,-1);
    \coordinate (3) at (-1,0);
    \coordinate (6) at (0,1);
    \coordinate (7) at (0,-1);
    \coordinate (10) at (1,0);
    \coordinate (11) at (2,1);
    \coordinate (12) at (2,-1);

    \draw (1) edge node [anchor=east]{$e_1$} (2);
    \draw (1) edge node [anchor=south west]{$e_2$} (3);
    \draw (2) edge node [anchor=north west]{$e_3$} (3);
    \draw (3) edge node [anchor=south east]{$e_4$} (6);
    \draw (3) edge node [anchor=north east]{$e_5$} (7);
    \draw (6) edge node [anchor=south west]{$e_6$} (10);
    \draw (7) edge node [anchor=north west]{$e_7$} (10);
    \draw (10) edge node [anchor=south east]{$e_8$} (11);
    \draw (10) edge node [anchor=north east]{$e_9$} (12);
    \draw (11) edge node [anchor=west]{$e_{10}$} (12);

    \draw [fill=black] (1) circle (0.03);
    \draw [fill=black] (2) circle (0.03);
    \draw [fill=black] (3) circle (0.03);
    \draw [fill=black] (6) circle (0.03);
    \draw [fill=black] (7) circle (0.03);
    \draw [fill=black] (10) circle (0.03);
    \draw [fill=black] (11) circle (0.03);
    \draw [fill=black] (12) circle (0.03);
\end{tikzpicture}
\end{center}
In this graph, the Graver basis consists of the following binomials:
\begin{align*}
  B_1 &= e_4e_7 - e_5e_6, & B_3 &= e_1e_5^2e_8e_9 - e_2e_3e_7^2e_{10},\\
  B_2 &= e_1e_4^2e_8e_9 - e_2e_3e_6^2e_{10}, & B_4 &= e_1e_4e_5e_8e_9 - e_2e_3e_6e_7e_{10}.
\end{align*}
\indent All four of these are elements of the Universal Gr\"obner Basis, and they are precisely the indispensible binomials of this graph. However, $B_4$ is not a circuit. In this case, the graph is robust, but there is a noncircuit indispensible binomial.
\end{remark}


\section{Applications to low-degree graph ideals}
The characterization given above takes on a particularly simple form in the case where the ideal is generated by quadratic binomials.
\begin{prop}
If $G$ is a simple graph and $I_G$ is robust and minimally generated by quadratics, then any two circuits of $G$ either share no edges or exactly two edges of opposite parity.
\end{prop}
\begin{proof}
  It suffices to restrict the characterization in Theorem \ref{charcircuits} to this special case. By Theorem \ref{urobustiffgrobust}, we know that all primitive elements are minimal generators; in particular, all the circuits of $G$ are even cycles of length $4$, which we will call ``squares''.
  \par Now suppose that there exist two squares of $G$ that share one or three edges. If they share three edges, then $G$ would not be simple, contradicting our assumptions on $G$. Next, suppose they share one edge. Then, our squares are as follows
  \begin{equation*}
    \begin{CD}
      c_1\colon v_1 @>>> v_2 @>f>> v_3 @>>> v_4 @>>> v_1\\
      c_2\colon w_1 @>>> v_2 @>f>> v_3 @>>> w_4 @>>> w_1
    \end{CD}
  \end{equation*}
  where $f$ is the shared edge, and the larger circuit
  \begin{equation*}
    \begin{CD}
      c_3\colon v_1 @>>> v_2 @>>> w_1 @>>> w_4 @>>> v_3 @>>> v_4 @>>> v_1
    \end{CD}
  \end{equation*}
  would produce a cubic generator, a contradiction.
  \par Now consider when two squares share two edges. If they share two edges of the same parity, then all $4$ of their vertices are the same, and they must be of one the following forms:
  \begin{equation*}
    \begin{CD}
      c_1\colon v_1 @>f>> v_2 @>>> v_3 @>f'>> v_4 @>>> v_1\\
      c_2\colon v_1 @>f>> v_2 @>>> v_3 @>f'>> v_4 @>>> v_1\\
    \end{CD}
  \end{equation*}
  or
  \begin{equation*}
    \begin{CD}
      c_1\colon v_1 @>f>> v_2 @>>> v_3 @>f'>> v_4 @>>> v_1\\
      c_2\colon v_1 @>f>> v_2 @>>> v_4 @>f'>> v_3 @>>> v_1\\
    \end{CD}
  \end{equation*}
  where $f$ and $f'$ are the shared edges. In the first case, our graph is not simple because there are two distinct edges connecting $v_2$ and $v_3$, and in the second case, the circuit $c_1$ has an effective crossing.
\end{proof}

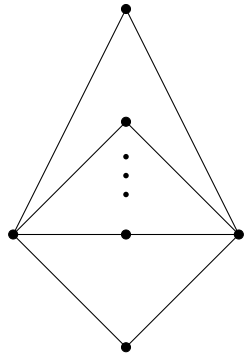
\begin{figure}[ht]
  \centering
    \begin{tikzpicture}[line cap=round,line join=round,x=1.5cm,y=1.5cm]
      \coordinate (l) at (1,0);
      \coordinate (c) at (0,0);
      \coordinate (r) at (-1,0);
      \coordinate (x1) at (0,2);
      \coordinate (x2) at (0,1);
      \coordinate (x3) at (0,-1);

      \draw (0,0.65) node {\scalebox{1.8}{$\vdots$}};

      \draw (l) -- (x1) -- (r);
      \draw (l) -- (x2) -- (r);
      \draw (l) -- (c) -- (r);
      \draw (l) -- (x3) -- (r);
      \draw[fill=black] (x1) circle (0.04);
      \draw[fill=black] (x2) circle (0.04);
      \draw[fill=black] (x3) circle (0.04);
      \draw[fill=black] (l) circle (0.04);
      \draw[fill=black] (c) circle (0.04);
      \draw[fill=black] (r) circle (0.04);
    \end{tikzpicture}
  \caption{All irreducible robust graphs with ideals generated by quadratics, up to isomorphism of graph ideals.}
  \label{fig:tents}
\end{figure}
\par Now recall the following result from \cite{BR}.
\begin{thm}[{\cite[Thm.~2.2]{BR}}]\label{quadchar}
  Let $F$ be an irreducible robust set consisting of irreducible quadratic binomials. Then, $F$ is robust if and only if $\lvert F\rvert = 1$ or $F$ consists of $2 \times 2$ minors of a generic $2 \times n$ matrix
  \begin{equation*}
    \begin{pmatrix}
      x_1 & \cdots & x_n\\
      y_1 & \cdots & y_n
    \end{pmatrix}
  \end{equation*}
  up to rescaling of variables.
\end{thm}
Since we can associate graphs that are of the type in Figure \ref{fig:tents} with the matrices in Theorem \ref{quadchar}, we have the following corollary
\begin{cor}
  All robust toric ideals generated by quadratics are graph ideals and come from the family of graphs given in Figure \ref{fig:tents}.
\end{cor}

Finally, to show the wide variety of possibilities for robust graphs, we have computed the set of connected robust graphs on seven vertices in Figure 9. To avoid trivialities, we assume no vertex has degree 1. The graphs are partitioned and labeled so that graphs that give isomorphic ideals are in the same partition.


\section{Concluding Remarks}

In light of Theorem \ref{urobustiffgrobust} it is natural to ask whether we can generalize this statement to toric ideals not arrising from graphs. Our proof of Theorem \ref{urobustiffgrobust} relied heavily on graph-theoretic arguments, but perhaps there is a more algebraic proof. Whether or not this theorem generalizes to all toric ideals remains an open question.
\begin{question}
If $I_A$ is a robust toric ideal, is $I_A$ minimally generated by its Graver basis?
\end{question}

\pagebreak

\section{Appendix}

\begin{figure}[ht]\label{graphson7}
\centering
\scalebox{0.44}{\input{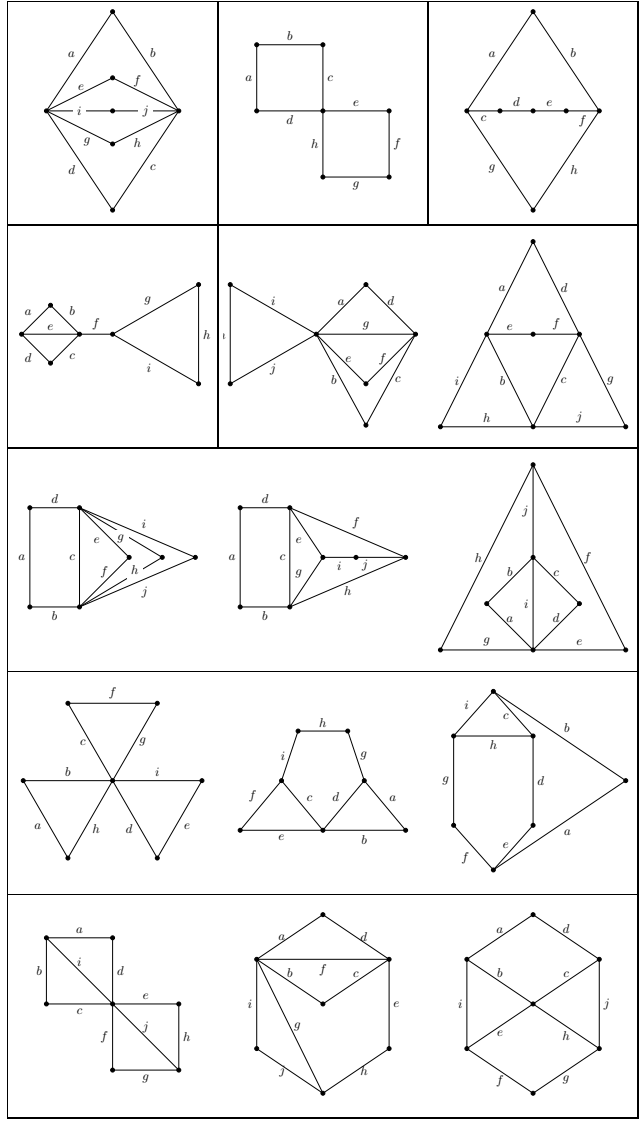}}
\caption{All connected robust graphs $G$ on 7 vertices such that the ideal $I_G$ has full support in its edge ring, divided up into isomorphism classes of $I_G$.}
\end{figure}

\pagebreak 

\section*{Acknowledgements}

This paper is the result of a summer research project conducted at the University of California, Berkeley.  This project was funded by the NSF award number 0838703.

\bibliographystyle{alpha}
\bibliography{graphbib}

\footnotesize {\bf Authors' addresses}:
\\ Department of Mathematics, University of California, Berkeley, CA. 
\hfill {\tt aboocher@math.berkeley.edu}
\\ Department of Mathematics, Pomona College, Claremont, CA. 
\hfill {\tt bcb02011@mymail.pomona.edu}
\\ Department of Mathematics, New College Florida, Sarasota, FL
\hfill {\tt timothy.duff@ncf.edu}
\\ Department of Mathematics, Reed College, Portland, OR
\hfill {\tt lymanla@reed.edu}
\\ Department of Mathematics, Princeton University, Princeton, NJ
\hfill {\tt takumim@princeton.edu}
\\ Department of Mathematics, Boston College, Chestnut Hill, MA
\hfill {\tt nesky@bc.edu}
\\ Department of Mathematics, Rice University, Houston, TX
\hfill {\tt kschaefer10@comcast.net}

\end{document}